\documentclass[reqno]{amsart}%
\usepackage{eurosym}
\usepackage{amsfonts,amsthm,amsmath,amssymb,amscd,amsmath,epsf,latexsym,mathrsfs}
\usepackage{graphicx}
\usepackage{amsmath}
\usepackage{amsfonts}
\usepackage{amssymb}%
\providecommand{\U}[1]{\protect\rule{.1in}{.1in}}
\providecommand{\U}[1]{\protect\rule{.1in}{.1in}}
\newtheorem{theorem}{Theorem}

\newtheorem{lemma}[theorem]{Lemma}

\newcommand{\bC}{\mathbb{C}}
\newcommand{\ee}{\end{equation}}

\newcommand \PP {\mathcal P}

\newcommand \ov {\overline}
\newcommand \Sym {\text{Sym}}
\begin{document}
\title[Topological classification of real meromorphic functions]{Topological classification of generic real meromorphic functions from compact surfaces}
\author[A.~F.~Costa]{Antonio F. Costa}
\address{Departamento de Matem\'aticas Fundamentales, Facultad de Ciencias, UNED, C.
Senda del rey, 9, 28040 Madrid, Spain}
\email{acosta@mat.uned.es}
\author[S.~Natanzon]{Sergey Natanzon}
\address{National Research University Higher School of Economics (HSE), 20 Myasnitskaya
ulitsa, Moscow 101000, Russia}
\email{natanzons@mail.ru}
\author[B.~Shapiro]{Boris Shapiro}
\address{Department of Mathematics, Stockholm University, SE-106 91 Stockholm, Sweden}
\email{shapiro@math.su.se}
\date{\today }

\begin{abstract}
In this article, to each generic real meromorphic function (i.e., having only
simple branch points in the appropriate sense) we associate a certain
combinatorial gadget which we call the park of a function. We show that the
park determines the topological type of the generic real meromorphic function
and the set of parks produce an stratification of the space of generic real
meromorphic functions. For any of the above topological types, we introduce
and calculate the corresponding Hurwitz number. Finally we relate the
topological types of generic real meromorphic functions with monodromy of
orbifold coverings.

\end{abstract}
\maketitle

\numberwithin{equation}{section}


\section{Introduction}

\label{sec1}

The main object of study in the present paper is the Hurwitz space
$\mathbb{R}H$ of real meromorphic functions on compact surfaces. Its structure
is substantially more complicated than that of the classical Hurwitz space $H$
of complex meromorphic functions. The latter space consists of pairs $(P,f)$
where $P$ is a complex algebraic curve, i.e., a compact orientable Riemann
surface and $f:P\rightarrow\ov {\mathbb{C}}$ is a meromorphic function, i.e.,
a holomorphic map to the Riemann sphere. Each point $(P,f)$ may be described
by polynomial equations with complex coefficients providing a natural topology
to the space $H$.

In applications one often encounters generic meromorphic functions, i.e.,
functions forming a dense subset $H^{0}\subset H$ which is stable with respect
to arbitrary small perturbations of functions. For complex meromorphic
functions, $H^{0}$ consists of functions with all simple {critical values}. In
other words, $H^{0}$ contains all functions $f$ whose branchings are of second
order and such that the images of all branchings under $f$ are pairwise
distinct, i.e., the critical values of $f$ are simple.

According to the basic classical results of Hurwitz \cite{Hu}, the space $H$
may stratified an each stratum consists of all functions of a given degree $d
$ on curves of a given genus $g$. (Such stratum is denoted by $H_{g,d}$.)
Analogously we can stratify $H^{0}$ by strata containing generic complex
functions of degree $d$ on curves of genus $g$. (denoted by $H_{g,d}^{0}$.)
Obviously, $H_{g,d}^{0}\subset H_{g,d}$ and one can show that the difference
$H_{g,d}\setminus H_{g,d}^{0}$ has real codimension two in $H_{g,d}$ which
explains why strata of $H$ are in $1-1$-correspondence with that of $H^{0}$.

Critical values of functions from $H_{g,d}^{0}$ provide the topology and
complex structure on $H_{g,d}^{0}$ (and on $H_{g,d}$ as well) induced by $H$.
The number of functions belonging to $H_{g,d}^{0}$ with a given set of
$2d+2g-2$ distinct critical values was defined and in many cases calculated by
Hurwitz in \cite{Hu}. (Observe that $2d+2g-2$ is also the dimension of
$H_{g,d}^{0}$ over $\mathbb{C}$.) The latter number of functions is
classically referred to as the corresponding \emph{Hurwitz number}. Hurwitz
numbers play important role in modern algebraic geometry and mathematical physics.

The space $H$ is endowed with a natural involution $\tau_{H}:H\rightarrow H$
anti-holomorphic with respect to the complex structure on $H$. This involution
associates to a pair $(P,f)$ the pair $(\ov P,\ov f)$, where the Riemann
surface $\ov P$ is obtained from $P$ by substituting all local charts on $P$
by their complex-conjugates and by using $\ov f(z)=\ov {f(z)}$. The set of
fixed points of $\tau_{H}$ coincides with the space $\mathbb{R}H$ of all real
meromorphic functions.

Let us now discuss the structure of $\mathbb{R} H$. By definition, for
$(P,f)\in\mathbb{R} H$, one has that $\tau_{H}(P,f)=(P,f)$. Therefore,
$\tau_{H}$ generates an anti-holomorphic involution $\tau: P\to P$ such that
$f (\tau(p))=\ov {f(p)}$. The category of such pairs $(P,\tau)$ is isomorphic
to the category of real algebraic curves, see \cite{AG}. Under this
isomorphism, real algebraic functions correspond to morphisms of real
algebraic curves to $(\ov {\mathbb{C}}, conj),$ where $conj(z)=\bar z$.

The set of strata of $\mathbb{R}H$ is substantially more complicated than that
of $H$, {see \cite{B,CR,CIR,N1,N2,N3,N4}}. (We will discuss these strata in
details in \S ~\ref{sec5}.) As in the complex case, the most important real
meromorphic functions for applications are generic ones; these functions are
the main object of study here.

As in the complex case, they form an open, dense and stable with respect to
arbitrary small deformations subset $\mathbb{R} H^{0}\subset\mathbb{R} H$. But
in the real situation the definition of $\mathbb{R} H^{0}$ becomes more
involved. In particular, $\mathbb{R} H^{0}$ does not coincide with $\mathbb{R}
H\cap H^{0}$ since the closure of the latter intersection is strictly smaller
than $\mathbb{R} H$. To improve the situation, one has to define $\mathbb{R}
H^{0}$ as the the set of all real meromorphic functions whose branchings are
of the second order and for the images of these branchings, i.e., for the
critical values, one should require that

\noindent(i) critical values are distinct when they are non-real;

\noindent(ii) critical values are the images of at most two branchings when
they are real.

In the latter case, when there are two branchings mapped to the same real
critical value, they must necessarily be interchanged by the involution $\tau$
on $P$. Therefore, the set $\mathbb{R}H^{0}$ defined in the above way is
stable under arbitrary small deformations inside the class of real meromorphic
functions. On the other hand, $\mathbb{R}H^{0}$ is dense in $\mathbb{R}H$.
Indeed, for any given real meromorphic function, using small perturbations
symmetric with respect to $\tau$, one can achieve that all branchings of its
small perturbations will be of order two and, therefore, belong to
$\mathbb{R}H^{0}$.

Contrary to the complex case, there are many more strata of the space
$\mathbb{R}H^{0}$ than in the space $\mathbb{R}H$ which, in the first place,
is related to the fact that $\mathbb{R}H\setminus\mathbb{R}H^{0}$ has real
codimension one. To describe the topological invariants distinguishing strata
of $\mathbb{R}H^{0}$, we have to introduce a rather complicated combinatorial
gadget which we call a \emph{park of a real meromorphic function}. This notion
is similar to the notion of a garden introduced in our earlier paper
\cite{NSW} and which describes strata in the space $\mathbb{R}H\cap
H_{0,d}^{0}$ of real generic rational functions. (The notion of a park is
substantially more complicated.) In \cite{NS} there is a classification of a
particular class of generic real meromorphic functions and now we shall
present the complete classification of these functions.

The structure of this paper is as follows. In \S ~\ref{sec2} we define a park
which, as we will see later, uniquely determines the topological type of
generic real meromorphic functions. In \S ~\ref{sec3} we explain how to
associate a park to an arbitrary generic real meromorphic function. In
\S ~\ref{sec4} we show that given an arbitrary abstract park $\mathcal{P}$, we
can always find a generic real meromorphic function whose associated park
coincides with $\mathcal{P}$. Let $H(\mathcal{P})$ the space of functions
corresponding to a given park $\mathcal{P}$. We show that $H(\mathcal{P})$ is
connected. In \S ~\ref{sec5} we recall some basic information about the
classification of real meromorphic functions, see \cite{N1,N2,N3} and explain
how to find the strata of the space $\mathbb{R}H$ coinciding with the sets
$H(\mathcal{P})$. In \S ~\ref{sec6} we define and calculate the corresponding
Hurwitz number for $H(\mathcal{P})$. Finally, in \S ~\ref{sec7} we explain how
parks are related to the monodromy of coverings of orbifolds.

\noindent\textbf{Acknowledgements.} The article was prepared within the
framework of the Academic Fund Program at the National Research University
Higher School of Economics (HSE) in 2017- 2018 (grant 17-01-0030) and by the
Russian Academic Excellence Project "5-100". The first author was supported by
the Spanish Ministry Project MTM2014-55812-P while the third author wants to
acknowledge the hospitality and the financial support of the Higher School of
Economics during his visit to Moscow in Spring 2014.


\numberwithin{equation}{section}


\section{Gardens and parks}

\label{sec2}

In order to motivate the definition we shall start with an example.

\textbf{Example 1}. Let $f:P\rightarrow\ov{\mathbb{C}}$ be a generic real
meromorphic function from a Riemann surface $P$ of genus $8$ such that $f$ has
two real critical values. Now we construct a combinatorial object that will
determine the topological type of $(P,f)$.

Consider in $\ov{\mathbb{C}}$ the hemispheres $\mathbb{C}_{+}$ and
$\mathbb{C}_{-}$ of complex numbers with positive and negative imaginary part
respectively. Let $U$ be a symmetric annular neighbourhood of the real cycle
$\ov {\mathbb{R}}$ which contains only real critical values. Its boundary
$\partial U$ consists of the contours $\mathcal{C}_{+}\subset\mathbb{C}_{+}$
and $\mathcal{C}_{-}\subset\mathbb{C}_{-}$. We shall call $D_{+}%
\subset\mathbb{C}_{+}$ and $D_{-}\subset\mathbb{C}_{-}$ the two connected
components of $\ov{\mathbb{C}}\setminus U$.

We consider the set $S=f^{-1}(\mathcal{C}_{+})\cup f^{-1}(\mathcal{C}_{-})$ of
simple contours in $P$. We cut the surface $P$ by the contours in $S$, we
assume that in our example the result is:

\begin{itemize}
\item a sphere with four holes $P_{R}=f^{-1}(U)$ ($f$ restricted to $P_{R}$ is
a two fold covering of the annulus $U$) and

\item two (possible non-connected) surfaces with boundary: $P^{+}=f^{-1}%
(D_{+})$ and $P^{-}=f^{-1}(D_{-})$.
\end{itemize}

The covering $f:P_{R}\rightarrow U$ can be extended to a real rational
meromorphic function $g$ which is described by an object called garden. In
this case the garden is given by the preimage by $g$ of $\ov {\mathbb{R}}$,
decomposing $\ov{\mathbb{C}}$ in four regions and a bicolouration of the four
regions (white for the preimage of $\mathbb{C}_{+}$ and black for the preimage
of $\mathbb{C}_{-}$.) In other cases there are other elements but that there
are not essential in our example: the enumeration of the preimage of real
critical values and some numbers related with the degree of $f$ on the
preimage of real points.%

\begin{figure}[ptb]%
\centering
\includegraphics[
height=2.1672in,
width=5.6853in
]%
{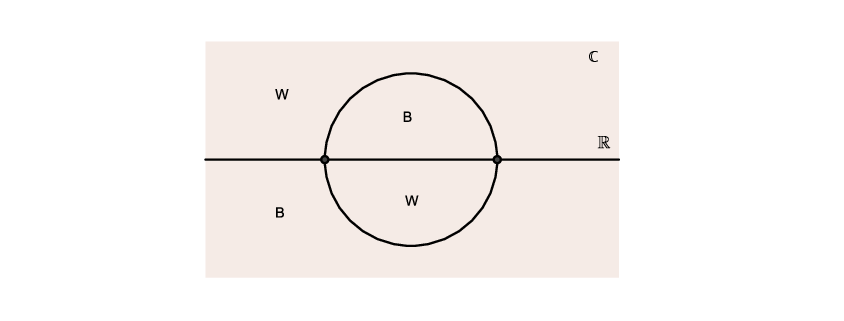}%
\caption{Garden of Example 1. The bicolouration of the faces is marked using
the letters W and B}%
\end{figure}

Our results will imply that the topological type of $(P,f)$ depends on the
number of connected components of $P^{+}$ and $P^{-}$ (the number of entrances
and exits), the genera of the components of $P^{+}$ and $P^{-}$ (weights on
entrances and exists), the way of connecting the boundaries of $P^{+}$ and
$P^{-}$ with the boundaries of $P_{R}$ (encoded by alleys connecting the white
regions with entrances and the black regions with exits.) For our example we
may assume one entrance and one exit of genus three that must be joined by
alleys as shown in the Figure 2.%

\begin{figure}[ptb]%
\centering
\includegraphics[
height=2.7138in,
width=4.5515in
]%
{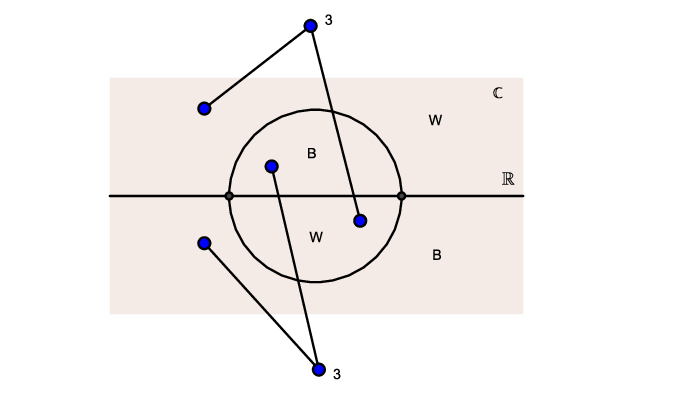}%
\caption{Park of Example 1}%
\end{figure}

\bigskip

To define our main combinatorial object: park, we need four preliminary
concepts: orientable garden, oriented separated garden, non-orientable garden
and non-oriented separated garden.

\subsection{Orientable gardens}

A \emph{semi-garden} $S\mathcal{G}$ is an structure composed by an oriented
surface, chords, simple contours, faces, edges, a colouration of faces,
lengths of edges and a cyclic order of the vertices. Now we describe each one
of such ingredients. First a compact {connected} oriented differentiable
surface $P$ with {or without} boundary together with a collection of
\textit{chords}, i.e., non-selfintersecting curves connecting points on the
boundary of $P$ and \textit{simple contours}, i.e., closed
non-selfintersecting curves embedded in $P$. Connected components of the
boundary of $P$, the chords and the contours of $S\mathcal{G}$ may intersect
transversely, but can not be tangent to each other. The set of chords and
simple contours splits $P$ into simply-connected domains called \emph{faces}.
Intersection points of chords and contours together with the endpoints of
chords are called \textit{vertices}. Vertices split the boundary of each face
into curve segments or closed curves called \emph{edges}. One additionally
requires that the end points of chords are pairwise distinct and that any
vertex which is an intersection point of chords and/or contours lies in the
closure of exactly four edges.

\medskip A semi-garden $S\mathcal{G}$ is additionally equipped with

\noindent(i) black and white colouring of the faces such that every two
neighbouring faces have different colours;

\noindent(ii) positive integers (\textquotedblleft lengths") assigned to its edges;

\noindent(iii) cyclic order on {the set of all vertices}.

\bigskip%

\begin{figure}[ptb]%
\centering
\includegraphics[
height=2.7138in,
width=4.619in
]%
{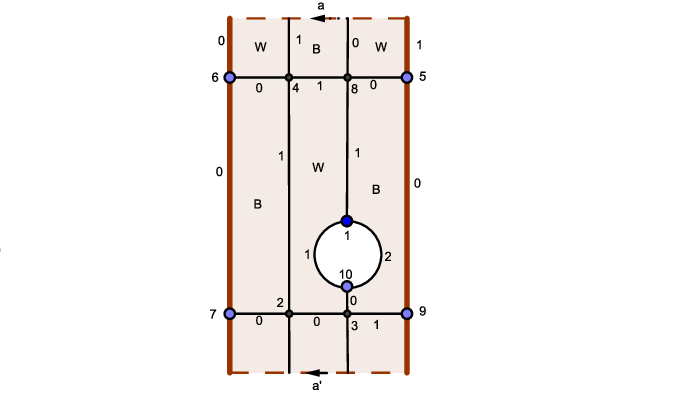}%
\caption{A semi-garden: $P$ has genus $0$ and three boundary components. The
sides $a$ and $a^{\prime}$ must be identified.}%
\end{figure}

Let $P^{\prime}$ be the surface $P$ with the opposite orientation. The
semi-garden obtained considering the surface $P^{\prime}$, the chords, simple
contours, lengths of edges and cyclic order of the edges of $S\mathcal{G}$ but
interchanging black and white colours of faces, is called the \emph{conjugate}
semi-garden $\overline{S\mathcal{G}}$ of $S\mathcal{G}$.

\bigskip

By an \emph{orientable garden} we call a pair of conjugate semi-gardens
$S\mathcal{G}$ and $\overline{S\mathcal{G}}$ {defined on surfaces with
non-empty boundary}. To an orientable garden we associate a Riemann surface
$\widetilde{P}$ obtained in the following way. Let $P$ and $P^{\prime}$ be the
surfaces corresponding to $S\mathcal{G}$ and $\overline{S\mathcal{G}}$
respectively, we consider $P$ and $P^{\prime}$ as disjoint surfaces with an
orientation reversing diffeomorphism $h:P\rightarrow P^{\prime}$ sending the
elements of the orientable garden $S\mathcal{G}$ to the elements of
$\overline{S\mathcal{G}}$ (the diffeomorphism $h$ sends black faces to white
faces) and we identify the points of the boundaries of $P$ with the ones of
$P^{\prime}$ by the map $h$. The surface $\widetilde{P}$ has an
orientation-reversing involution $\nu$ (given by $h$ and $h^{-1}$) sending the
initial semi-garden ot its conjugate. The quotient surface $\widetilde{P}/\nu$
is an orientable surface isomorphic to $P$. Thus an orientable garden is given
by a compact orientable surface $\widetilde{P}$ without boundary with an
anticonformal involution $\nu$ such that $\widetilde{P}/\nu$ is orientable and
an additional structure provided by a semi-garden on $\widetilde{P}/\nu$.%

\begin{figure}[ptb]%
\centering
\includegraphics[
height=2.5028in,
width=5.2918in
]%
{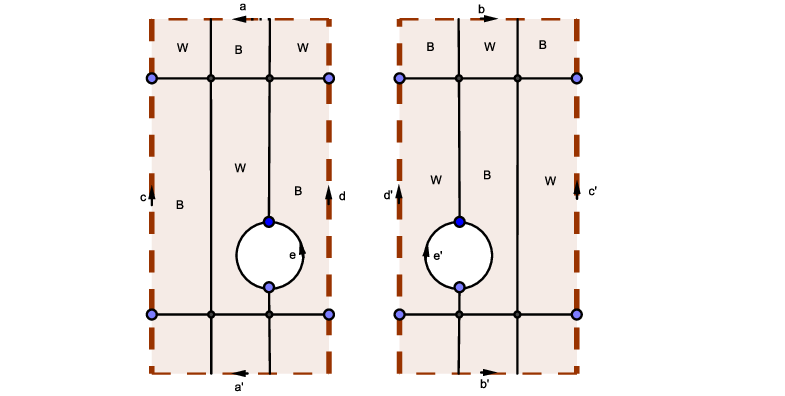}%
\caption{An orientable garden. To obtain $\protect\widetilde{P}$ it is
necessary to identify $a$ with $a^{\prime}$, $b$ with $b^{\prime}$, $c$ with
$c^{\prime}$ and $e$ with $e^{\prime}$, producing a surface of genus $2$.}%
\end{figure}

In Figure 4 appears an orientable garden associated to a genus 2 surface, the
corresponding semigardens are the one represented in Figure 3 and its conjugate.

\subsection{Oriented separated garden}

An \textit{elementary orientable garden} is given by the Riemann sphere with
the standard complex conjugation $\sigma$ and a given integer
\textquotedblleft length\textquotedblright\ of the real cycle (the
corresponding semi-gardens have neither chords nor simple contours.) Now
consider a pair of elementary orientable gardens $\{(\mathcal{G}_{i}%
,\sigma_{i}):i=0,1\}$ with equal lengths of their real cycles. Take the
identical isomorphisms $\phi_{i}:(\mathcal{G}_{i},\sigma_{i})\rightarrow
(\mathcal{G}_{1-i},\sigma_{1-i})$ and the involution $\sigma:=\phi_{i}%
\sigma_{i}$ of $\mathcal{G}_{0}\cup\mathcal{G}_{1}$.
Call the obtained structure $(\mathcal{G}_{0}\cup\mathcal{G}_{1},\sigma)$ an
\emph{oriented separated garden}.

\subsection{Non-orientable gardens}

{A \textit{symmetric semi-garden} $SS\mathcal{G}$ is a pair consisting of a
semi-garden on a surface $P$ {without cyclic enumeration of its vertices} and
an orientation-reversing involution $\tau:P\rightarrow P$ {without fixed
points} which changes the colour of the faces and preserves the
\textquotedblleft lengths\textquotedblright\ of the edges. In{\ a symmetric
semi-garden $SS\mathcal{G}$ the cyclic enumeration of the vertices of a
semi-garden is substituted by the cyclic numeration of pairs of vertices which
are interchanged by the involution $\tau$.} }

A \emph{non-orientable garden} is {a symmetric semi-garden $SS\mathcal{G}{\ }
$and an associated orientable surface }$\widetilde{P}${\ obtained identifying
all pairs of points on the boundary of $P$ interchanged by $\tau$. The
involution {$\tau$ induces an orientation reversing involution }%
}$\widetilde{{\tau}}$ {on }$\widetilde{P}${\ and }$\mathrm{Fix}%
(\widetilde{{\tau}})$ does not separates $\widetilde{P}$, then $\widetilde{P}%
/\widetilde{{\tau}}$ is non-orientable.

Thus a non-orientable garden is given by a compact connected orientable
surface $\widetilde{P}$ without boundary, an orientation reversing involution
$\widetilde{{\tau}}$ such that $\widetilde{P}/\widetilde{{\tau}} $ is
non-orientable and with a symmetric semi-garden structure on $\overline
{\widetilde{P}-\mathrm{Fix}(\widetilde{{\tau}})}$. Note that $\mathrm{Fix}%
(\widetilde{{\tau}})$ may be empty and in this case {$P$ has no boundary.}%

\begin{figure}[ptb]%
\centering
\includegraphics[
height=2.3756in,
width=4.7167in
]%
{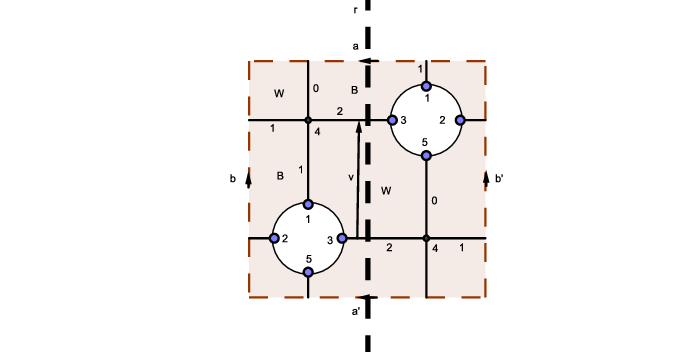}%
\caption{A symmetric semi-garden. The sides $a$ and $a\prime$, $b$ and
$b\prime$ must be identified to obtain a torus with two boundary components.}%
\end{figure}

The Figure 5 shows a symmetric semi-garden corresponding to a surface $P$ of
genus one with two boundary components. The orientation reversing involution
of $P$ is given by the glide-reflection of axis $r$ and vector $v$. The
non-orientable garden given by the symmetric semi-garden has associated a
surface of genus two and the orientation reversing involution has one contour
as fixed point set (one oval).

\subsection{Non-oriented separated garden}

A \textit{elementary non-orientable garden} is the Riemann sphere with the
antipodal involution $\tau:z\rightarrow-\bar{z}^{-1}$, a given integer
\textquotedblleft length\textquotedblright\ of the real cycle, and a given
colouring of the hemispheres into which the real cycle splits the Riemann
sphere. Consider a pair of elementary non-orientable gardens $\{(\mathcal{G}%
_{i},\tau_{i})|\;i=0,1\}$ with equal lengths of their real cycles. Take the
identical isomorphisms $\phi_{i}:(\mathcal{G}_{i},\sigma_{i})\rightarrow
(\mathcal{G}_{1-i},\sigma_{1-i})$ and the involution $\sigma:=\phi_{i}\tau
_{i}$ of $\mathcal{G}_{0}\cup\mathcal{G}_{1}$.
The obtained structure $(\mathcal{G}_{0}\cup\mathcal{G}_{1},\sigma)$ is called
a \textit{non-oriented separated garden}.

A separated gardens is an oriented or non-oriented separated garden.

\subsection{Parks}

A \emph{park} is a structure consisting of

\noindent(i) (orientable and non-orientable) gardens and (oriented and
non-oriented) separated gardens;

\noindent(ii) entrances and exits;

\noindent(iii) alleys,

\noindent which we shall define below.

An \emph{entrance} (resp. \emph{exit}) is a point coloured white (resp. black)
and endowed with a non-negative integer called \emph{weight}. By an
\emph{alley} we call a path connecting an entrance or an exit with a face of a
garden or a separated garden.

\medskip A \emph{park} satisfies the following conditions:

\noindent(i) alleys connect entrances to white faces and exits with black faces;

\noindent(ii) each face has exactly one alley;

\noindent(iii) at least one alley is connected to each entrance or exit;

\noindent(iv) there exists an involution sending entrances to exits
(preserving weights), alleys to themselves and acting on gardens and separated
gardens as their associated involutions.

Note that the information on the number of exits, alleys from exits and its
weights may be deduced from the corresponding concepts on entrances and
vice-versa. We maintain the two types of elements by aesthetical reasons.

Two parks are called \emph{isomorphic} if there exist homeomorphisms of the
associated surfaces to the gardens and bijections between the sets of
entrances and exits preserving all elements and relations of the structures.

The Figure 6 shows a park constructed from two non-orientable gardens as the
one in Figure 5. In the Example 3 of the next section we describe the generic
real meromorphic functions with the topological type represented by this park.%

\begin{figure}[ptb]%
\centering
\includegraphics[
height=2.3756in,
width=4.7167in
]%
{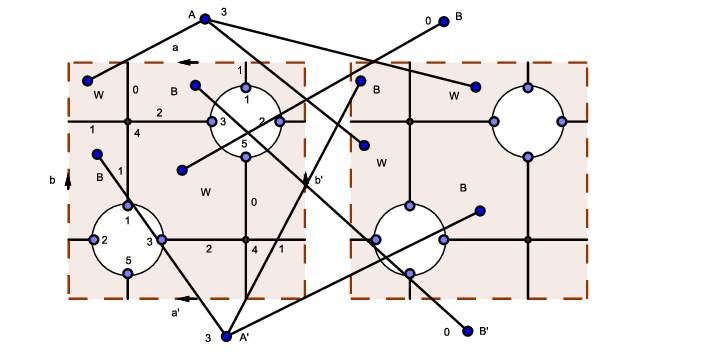}%
\caption{An example of park}%
\end{figure}


\section{Park corresponding to a real meromorphic function}

\label{sec3}

Recall that a \emph{real algebraic curve} is a pair $(P,\tau)$ consisting of a
compact Riemann surface $P$ and an anti-holomorphic involution $\tau: P\to P$.
A real meromorphic function on $(P,\tau)$ is a holomorphic function $f:
P\to\ov{\mathbb{C}}$ such that for any $p\in P$, $f(\tau(p))=\ov{f(p)}$. A
real meromorphic function $f$ is called \emph{generic} if (i) all critical
points $p\in P$ have degree 2, and (ii) {if critical points $p_{1}\neq
p_{2}\in P $ have the same image $z=f(p_{1})=f(p_{2})$, then $z\in
\overline{\mathbb{R}}$ and $\tau(p_{1})=p_{2}$.}

\medskip Let us now associate a park to any generic real meromorphic function.

\subsection{Associated gardens}

The real cycle $\ov{\mathbb{R}}\subset\ov{\mathbb{C}}$ splits the Riemann
sphere into the hemispheres $\mathbb{C}_{+}$ and $\mathbb{C}_{-}$ with
positive and negative imaginary part respectively. We will call them resp. the
\emph{white} and the \emph{black} hemispheres.

Consider a symmetric neighbourhood of a real cycle $U\supset\ov {\mathbb{R}}$
(i.e., $\ov {U}=U$) which contains only real critical values. Its boundary
$\partial U$ consists of the contours $\mathcal{C}_{+}\subset\mathbb{C}_{+}$
and $\mathcal{C}_{-}\subset\mathbb{C}_{-}$.

Take the preimage $f^{-1}(U)$. Contracting each connected component of {the
boundary of the preimage} to a point, we obtain a finite family $\Phi$ of
{compact} connected Riemann surfaces. The preimage $f^{-1}(\ov {\mathbb{R}})$
splits each of them into simply-connected domains. {The closures of these
domains are called \textit{faces}.} Each face contains exactly one point
obtained as the result of the contraction of a {component in $f^{-1}%
(\partial{\ U})$}. Let us colour the face white if this point is obtained by
contracting of the inverse image of $f^{-1}(\mathcal{C}_{+})$ and black if it
is obtained by contracting $f^{-1}(\mathcal{C}_{-}).$

{Contracting} each of the contours $\mathcal{C}_{\pm}$ to a point, we get a
sphere $\ov U$ with two labelled points $z_{\pm}$. The real cycle splits it
into the white and the black hemispheres. Function $f$ induces the branched
covering $\bar f$ of $\ov U$ by the family $\Phi$ which maps white faces to
the white hemisphere and black faces to the black one. Points $z_{\pm}$ will
be the only {\ non-real critical values}.

{A critical point with a real value will be called a \textit{vertex}. The
cyclic order of critical values generates a cyclic enumeration of the vertices
(or pair of vertices in the case of non-orientable gardens). Observe that
vertices lie on the union of the boundaries of faces and they divide the
boundary of each face into curve segments called \textit{edges}.} The
restriction of $f$ to any edge $r$ has no critical points and it maps $r$ in
$\ov{\mathbb{R}}.$ The degree of this restricted map, i.e., the minimal
cardinality of the preimage of a point of $\ov{\mathbb{R}}$ by $f$ is called
the {\textquotedblleft\textit{length}"} of the edge $r$.

Involution $\tau$ sends the family $\Phi$ to itself interchanging the colours
of all faces and preserving the rest of the structure.

\begin{lemma}
\label{lm1} If a surface $P\in\Phi$ is invariant under $\tau$ (except for the
colour change), then it forms a garden. The rest of components split into
pairs of separated gardens.
\end{lemma}

\subsection{Associated entrances, exits and alleys}

By an \emph{entrance} (resp. \emph{exit}) we call a connected component of the
inverse image $f^{-1}( \mathbb{C}_{+}\setminus U)$ (resp. $f^{-1}(
\mathbb{C}_{-}\setminus U)$). The weight of an entrance (resp. exit) is the
genus of this connected component. Connected component of the inverse image
$f^{-1}(\partial U)$ is a contour separating an entrance or an exit from a
face of a garden or a separated garden of the same colour. We assign to this
connected component its alley (i.e., a path) connecting the entrance/exit and
the face which it separates.

\medskip The next statement is obvious.

\begin{theorem}
\label{th1} Gardens, separated gardens, and alleys of a generic meromorphic
function $f$ form a park $\PP(f)$ which we call the park of $f$. This park is
determined uniquely up to an isomorphism.
\end{theorem}

\smallskip\noindent\textbf{Example 2.} To illustrate the two types of separate
gardens, we give an example of a generic real meromorphic function and its
associated park. Take a generic meromorphic function of degree four on
$(T,\tau),$ where $T$ is a torus and $\tau$ is an anti-conformal involution of
$T$ without fixed points (i.e., $T/\left\langle \tau\right\rangle $ is a Klein
bottle). Observe that such a meromorphic function has $8$ {critical values}.
Its park consists of:

\noindent(i) two separated gardens: one oriented $(\mathcal{G}_{0}^{or}%
\cup\mathcal{G}_{1}^{or},\sigma)$ and one non-oriented $(\mathcal{G}_{0}%
^{nor}\cup\mathcal{G}_{1}^{nor},\tau)$;

\noindent(ii) two entrances $N_{1}$, $N_{2}$ and two exits $X_{1}$, $X_{2}$,
all four entrances/exits being of genus $0$.

\noindent(iii) Denote by $B\mathcal{G}_{i}^{or}$, $B\mathcal{G}_{i}^{nor}$ and
$W\mathcal{G}_{i}^{or}$, $W\mathcal{G}_{i}^{nor}$ the black and the white
faces of the separated gardens respectively. There are three alleys from
$N_{1}$ to $W\mathcal{G}_{0}^{nor}$, $W\mathcal{G}_{1}^{nor}$ and
$W\mathcal{G}_{0}^{or}$ and three alleys from $X_{1}$ to $B\mathcal{G}%
_{0}^{nor}$, $B\mathcal{G}_{1}^{nor}$ and $B\mathcal{G}_{1}^{or}$, one alley
from $X_{2}$ to $B\mathcal{G}_{0}^{or}$ and one alley from $N_{2}$ to
$W\mathcal{G}_{1}^{or}$.

The involution of the park sends $N_{i}$ to $X_{i}$ and%

\[
(W\mathcal{G}_{i}^{or},B\mathcal{G}_{i}^{or},W\mathcal{G}_{i}^{nor}%
,B\mathcal{G}_{i}^{nor})\rightarrow(B\mathcal{G}_{1-i}^{or},W\mathcal{G}%
_{1-i}^{or},B\mathcal{G}_{1-i}^{nor},W\mathcal{G}_{1-i}^{nor}).
\]

(See Figure 7.)%

\begin{figure}[ptb]%
\centering
\includegraphics[
height=2.4474in,
width=5.1802in
]%
{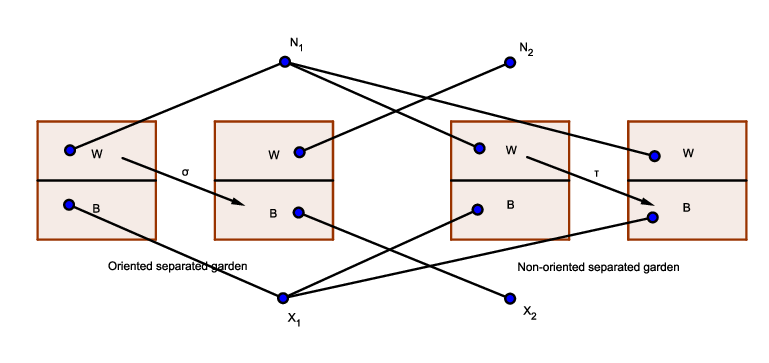}%
\caption{A park with two separated gardens}%
\end{figure}

\medskip\noindent\textbf{Example 3.} Consider a generic rational function of
degree $3$. There is a unique topological type of such functions with four
critical values. Let us describe some topological types of real meromorphic
functions in the Hurwitz space $H_{0,3}$. There are three possibilities.
Namely, either all the critical values lie on the real cycle, or only two
critical values are on the real cycle, or, finally, there are no real critical
values at all.

A park for a generic real meromorphic function with four real critical values
consists of a garden which is a sphere with two conjugate semi-gardens being
disks with two chords. There are six faces (three white and three black), each
one with an alley to an entrance/exit of genus $0$ (see Figure 8.)%

\begin{figure}[ptb]%
\centering
\includegraphics[
height=2.9784in,
width=6.2993in
]%
{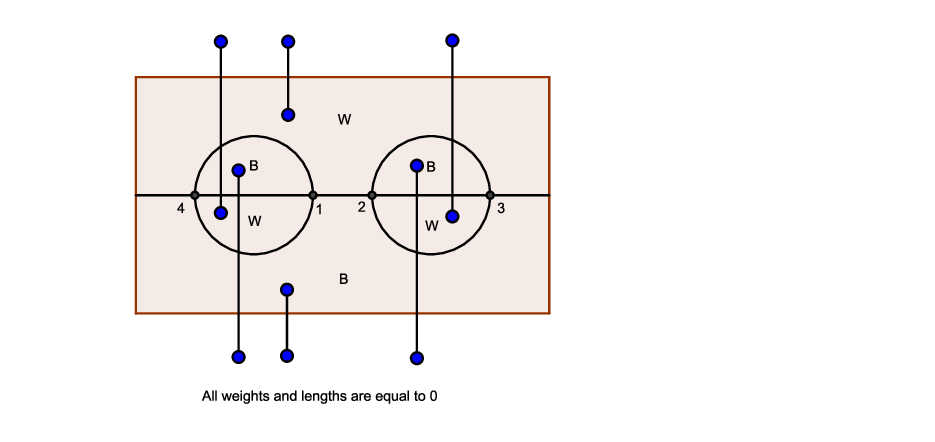}%
\caption{Example of a generic real rational function of degree three }%
\end{figure}

A park of generic real meromorphic functions with two real critical values is
as follows. Its semi-garden consists of a disk and a chord joining two points
on the boundary. The corresponding orientable surface of the orientable garden
is a sphere. There is an alley in each face (there are four faces) ending at
an entrance/exit of genus $0$. The lengths of the three edges of each
semi-garden are $0,0$ and $1$.

Finally, a park of a generic real meromorphic function without real critical
values has a garden which is the elementary orientable garden with length $3$
of its side. It has one entrance/exit, each one of genus $0$, joined by an
alley with the corresponding white/black face. A different park for this case
may be constructed with a elementary orientable garden, an oriented separated
garden and four entrances/exits all of them of genus $0$. The edges are three
contours of length $1$.

\medskip\noindent\textbf{Example 4.} The park in Figure 6 corresponds to
generic real meromorphic functions from surfaces of genus $13$, with $10$ real
critical values. The orientation reversing involution has two Jordan closed
curves as fixed point set and the quotient surface is non-orientable.

\section{Real meromorphic function corresponding to a park}

\label{sec4}

The main goal of this section is to prove the following result.

\begin{theorem}
\label{th2} For any park $\mathcal{P}$, there exists a generic real
meromorphic function $f$ whose park $\mathcal{P}(f)$ is isomorphic to
$\mathcal{P}$.
\end{theorem}

\begin{proof}
For any finite branched covering $f: P \to\ov {\mathbb{C}}$ of the Riemann
sphere $\ov {\mathbb{C}}$ by an orientable closed $2$-manifold $P$, there
exists a unique holomorphic structure on $P$ such that $f$ becomes a
meromorphic function, see \cite{K}. Analogously, for any orientation-reversing
involution $\tau: P \to P$ of an orientable closed $2$-manifold $P$ and a
finite branched covering $f: P\to\ov {\mathbb{C}}$ such that $f(\tau
(p))=\ov f(p)$ for any $p\in P$, there exists a unique complex structure on
$P$ which makes $f$ a real meromorphic function, see \cite{N2,N3}. Thus to
obtain a generic real meromorphic function realizing a given park
$\mathcal{P},$ it suffices to construct a topological model of this function.

\medskip We shall make the construction for parks with orientable gardens. The
case for non-orientable gardens is similar.

\medskip We start with a topological model of a function realizing a garden
{with $n$ vertices}. {Consider the Riemann sphere with $n$ points on its real
cycle with cyclic enumeration. The real cycle divides the Riemann sphere into
a white and a black disks. By definition, a garden is a $2$-dimensional
surface split into simply-connected {white and black} faces.} The boundary of
each face is split by labelled vertices into edges of a given integer
\textquotedblleft length\textquotedblright. We will construct a branched
covering, {which maps any face to the disk of the same colour and has at most
one critical point inside each face}. To obtain such a covering, it suffices
to determine the mapping of the boundary of each face to the real cycle
$\ov{\mathbb{R}}$. (This mapping of the boundary is uniquely determined by the
coincidence of the enumerations of the vertices of the face with the cyclic
enumeration of points on $\mathbb{R}$ together with the coincidence of
\textquotedblleft length\textquotedblright\ of every edge with the degree of
the mapping on this edge.)

Mappings of elementary and separated gardens are constructed similarly. By the
\emph{degree} of a face we call the degree of the resulting mapping of this
face to a disk. The latter degree will be called the \emph{degree of the
alley} related to the face under consideration. This degree may be read from
the park as the sum of the lengths of the edges of the face where the alley ends.

Now consider a neighbourhood of the real cycle $\ov {\mathbb{R}}\subset
U\subset\ov{\mathbb{C}}$ symmetric with respect to the complex conjugation and
such {that it contains only real critical values of the constructed map}. Its
complement \textbf{$\mathbb{C}\setminus U$} splits into the {white and black}
disks $D_{+}$ and $D_{-}$. The preimage of $U$ with respect to the constructed
branched covering consists of gardens such that every face has exactly one
hole. We call them \emph{gardens with boundary}.

Now we shall construct maps corresponding to entrances and alleys starting at
them. Associate to an entrance its \emph{entrance surface} whose genus equals
to the weight of the entrance and the number of holes equals to the number of
alleys. Consider a generic covering of this surface over $D_{+}$ which maps
the boundary components corresponding to the alleys to $\partial D_{+}$ in
such a way that the degree of the (unbranched) covering equals the degree of
the alley. The number of critical points of such a covering (the latter number
being determined by the Riemann-Hurwitz formula) is called the \emph{index of
the entrance}.

The boundary points of gardens with boundary and the boundary points of
entrance surfaces are called \emph{related} if their images on $\partial
D_{+}$ with respect to constructed coverings coincide. {Now glue together the
boundary components of the alleys with the corresponding boundary components
of the gardens with boundary by homeomorphisms which identify related points.}
{Using the involution, we do the same with the exit surface.}

As a result, we obtain a branched covering with an orientation-reversing
involution whose corresponding real meromorphic function has a park isomorphic
to $\mathcal{P}$.
\end{proof}

\section{Parks as topological types of generic real meromorphic functions}

\label{sec5}

We now settle the following statement.

\begin{theorem}
\label{th:3} Two generic real meromorphic functions with isomorphic parks are
topologically equivalent. The set $H(\mathcal{P})$ of all generic real
meromorphic functions with a given (up to isomorphism) park $\mathcal{P}$ is a
single connected component in the space of generic meromorphic functions. In
other words, in the space $\mathbb{R} H^{0}$ of generic real meromorphic
functions there exists a deformation transforming one generic real meromorphic
function with a given park $\mathcal{P}$ to any other generic real meromorphic
function with the same park.
\end{theorem}

\begin{proof}
{We say that two generic real meromorphic functions are \textit{related} if
one can deform one to the other by means of a continuous change of real
critical values and pairs of complex conjugated non-real critical values. A
real meromorphic function defined by a disjoint union of gardens and separated
gardens is called \textit{simple}. Construction in \S ~\ref{sec4} associates
to any real meromorphic function $f$ a related class of simple real
meromorphic functions defined by gardens and separated gardens of the park of
$f$. Moreover, this related class is the same for two given functions if and
only if they have parks with isomorphic systems of gardens and separated
gardens.}

{A system of entrances and its alleys of a generic real meromorphic function
determines a set of ramified coverings with simple critical values
$\varphi_{i}:P_{i}\rightarrow D$ of degrees $d_{i}$, where $P_{i}$ is a
Riemann surface of genus $g_{i}$ with $k_{i}$ holes and $D$ is a disk.
Consider the degrees $d^{1}_{i},\dots,d^{k_{i}}_{i}$ of $\varphi
_{i}|_{\partial P_{i}}:\partial P_{i}\rightarrow\partial D$ on components
$c_{j}\subset\partial P_{i}$. It follows from \cite{N4,N2} and Chapter 3 of
\cite{N3} that any other covering $\varphi_{i}^{\prime}:P_{i}^{\prime
}\rightarrow D$ is obtained from $\varphi_{i}$ by a continuous change of
critical values if the topological types $(g_{i},k_{i}) $ and $(g_{i}^{\prime
},k_{i}^{\prime})$ respectively of $P_{i}$ and $P_{i}^{\prime} $ are the same,
the degrees $d_{i}$ and $d_{i}^{\prime}$ of $\varphi_{i}$ and $\varphi
_{i}^{\prime}$ are the same, and the degrees $d^{1}_{i},\dots,d^{k_{i}}_{i}$
and $d^{\prime1_{i}},\dots,d^{\prime k_{i}}_{i}$ on the boundary components
are the same.}

{The park of a generic real function defines the topological type
$\{(g_{i},k_{i})\}$ of the system of entrances and its alleys and the
corresponding degrees $d^{1}_{i},\dots,d^{k_{i}}_{i}$. Thus for functions with
isomorphic parks, there exists a continuous change of critical values
deforming one system of entrances and its alleys to the other. The
anti-holomorphic involution generates a continuous change of critical values
deforming one system of exits and its alleys to the other. Together with the
change of real critical values which we considered above, this generates a
continuous change of critical values transforming one real function to the
other.}
\end{proof}

\section{Hurwitz numbers}

\label{sec6}

{A Hurwitz number is, by definition, a weighted number of equivalence classes
of coverings with fixed critical values. Generic real meromorphic functions
belong to different equivalence classes if and only if they have different
topological types (i.e., non-equivalent parks). Thus it is enough to find a
Hurwitz number for any given topological type.}

{Consider the set $H(\mathcal{P})$ of all generic real meromorphic functions
with a given (up to an isomorphism) park $\mathcal{P}$. In \S ~\ref{sec4} we
constructed a decomposition of a generic real meromorphic function into a real
function defined on the system of gardens and separated gardens and a function
defined on surfaces generated by of entrances and exits with alleys. Thus the
the Hurwitz number for $H(\mathcal{P})$ is obtained by multiplication of the
Hurwitz number for the function corresponding to the gardens and separated
gardens by the Hurwitz number for the function defined on the surfaces
generated by the entrances and its alleys. (Notice that the function defined
on the surfaces generated by the entrances and its alleys determines the
function defined on the surfaces generated by the exits and its alleys.)}

{Main construction from \S ~\ref{sec4} demonstrates that the Hurwitz number
for functions corresponding to the gardens and separated gardens equals 1.
This fact means that to calculate the total Hurwitz number it suffices to
determine the Hurwitz number for the function defined on the system of
surfaces generated by the entrances and its alleys. This function, in its
turn, is a union of functions with simple critical values on surfaces $P_{i}$
of topological types $\{(g_{i},k_{i}):i=1,\dots,n\}$ and degrees $\{d_{i}%
^{j}:j=1,\dots,k_{i}\}$ on boundary components, see \S ~\ref{sec5}.}

{Collapsing the boundary components of $P_{i}$ and the boundary of the image
disk to a point, we obtain a map of the closed surface $\overline{P}_{i}$ to
the Riemann sphere which have only one non-simple critical value. The number
of its simple critical values is calculated by the Riemann-Hurwitz formula and
is equal to
\[
b_{i}=2g_{i}-2+k_{i}+\sum_{j=1}^{k_{i}}d_{i}^{j}.
\]
The Hurwitz number $H_{g_{i}}(d_{i}^{1},\dots,d_{i}^{k_{i}})$ for such a map
is equal to the Hurwitz number corresponding to the coverings of $P_{i}$ over
a disk and is called a \textit{single Hurwitz number}. These numbers appear in
different areas of mathematics from moduli spaces to integrable systems and
are well studied.}

{Now the Hurwitz number for the topological type given by all entrances and
its alleys is:
\[
\frac{(b_{1}+\dots+b_{n})!}{b_{1}!\dots b_{k}!}\prod_{i=1}^{n}H_{g_{i}}%
(d_{i}^{1},\dots,d_{i}^{k_{i}}).
\]
As we explained above, this number coincides with the Hurwitz number for
generic real meromorphic functions of type $\mathcal{P}$.}

\textbf{Example 5}. Consider the park $\mathcal{P}$ in Figure 2 where all the
lengths of sides are $0$. There is only one entrance of genus $3$ with two
boundary components of degree $1$. The number $b$ is $8$ and the Hurwitz
number for the topological type of generic real meromorphic functions given by
$\mathcal{P}$ is:%
\[
H(\mathcal{P)=}\frac{8!}{8!}H_{3}(1,1)
\]

\vspace{2ex}

\section{Parks and monodromy of coverings}

\label{sec7}

The traditional way to describe coverings and meromorphic functions is by
using their monodromies. In the case of a degree $d$ generic meromorphic
function $f:P\rightarrow\ov{\mathbb{C}}$, the monodromy is a representation
$\omega:\pi_{1}(\ov{\bC}\setminus\{z_{1},...,z_{n}\})\rightarrow\Sym_{d}$,
where $\{z_{1},...,z_{n}\}$ is the set of critical values and $\Sym_{d}$ is
the symmetric group of permutations of $d$ elements.
For a generic meromorphic function, $\omega$ sends each based small loop
encircling a singular value to a transposition. The monodromy tell us how the
elements of $\pi_{1}(\ov{\bC}\setminus\{z_{1},...,z_{n}\})$ lift to the
surface $P$ and $\omega^{-1}(Stab(1))$ is isomorphic to the group $\pi
_{1}(P\setminus f^{-1}\{z_{1},...,z_{n}\})$. Two monodromy representations
correspond to topologically equivalent meromorphic functions if and only if
there exists an automorphism of $\pi_{1}(\ov {\mathbb{C}}\setminus
\{z_{1},...,z_{n}\})$ (given by an orientation preserving homeomorphism) and
an automorphism of $\Sym_{d}$ sending one representation to the other. In this
way, it is possible to prove that there exists a unique topological type of
degree $d$ generic meromorphic functions from Riemann surfaces of a fixed
genus (see \cite{C} and \cite{EEHS}).

The degree $d$ generic real meromorphic functions may be described by a degree
$d$ orbifold coverings $\widehat{f}:\widehat{P}\rightarrow D_{t,\overline{s}}$
where $D_{t,\overline{s}}$ is an orbifold with the topological type of a disc,
$t$ conic points of order two and $s$ corner points of angle $\pi/2$, in such
a way that if $P$ is the complex double of $P$ then $f:P\rightarrow
\ov {\mathbb{C}}$ is a generic meromorphic function with $2t+s=n$ critical
values, $s$ of them real. These coverings are given by a monodromy
representation:%
\[
\varpi:\pi_{1}O(D_{t,\overline{s}})\rightarrow\Sym_{d}=\Sym\{1,...,d\}.
\]

The group $\pi_{1}O(D_{t,\overline{s}})$ has a presentation:%
\[
\left\langle x_{1},...,x_{t},e,c_{1},...,c_{s+1}:\;x_{i}^{2}=c_{i}^{2}%
=(c_{i}c_{i+1})^{2}=1;\;x_{1}\dots x_{t}e=1;\;c_{1}=ec_{t+1}e^{-1}%
\right\rangle .
\]

Since $\varpi$ is given by a generic real meromorphic function, it satisfies
some specific properties:

\noindent1. the permutations $\varpi(x_{i})$ are transpositions;

\noindent2. the permutations $\varpi(c_{i}c_{i+1\operatorname{mod}s})$ are
transpositions or products of two (disjoint)\ transpositions;

\noindent3. every $\varpi(c_{i})$ is an order two element.

The entrances of the park are given by the orbits of the subgroup generated
by
\[
\left\langle \varpi(x_{1}),\dots,\varpi(x_{t})\right\rangle
\]
The white faces and the corresponding alleys are in $1-1$ correspondence with
the orbits of the element $\varpi(e)$. The face corresponding to an orbit of
$\varpi(e)$ is joined with the entrance which is the orbit of $\left\langle
\varpi(x_{1}),\dots,\varpi(x_{t})\right\rangle $ containing it.

The exits, black faces and corresponding alleys are obtained using the orbits
of
\[
\left\langle \varpi(c_{1}x_{1}c_{1}),\dots,\varpi(c_{1}x_{t}c_{1}%
)\right\rangle
\]
and the cycles of $\varpi(c_{1}ec_{1})$.

\bigskip

\textbf{Example 6}. We can apply Theorem 4 to know if two generic real
meromorphic functions given by its monodromies have or not the same
topological type. For example, let $f_{1}:$ $P_{1}\rightarrow\ov{\mathbb{C}} $
and $f_{2}:P_{2}\rightarrow\ov{\mathbb{C}}$ be two generic real meromorphic
functions of degree $d$ without real critical values from surfaces of a fixed
genus $g\geq2$. Let
\[
\varpi_{i}:\pi_{1}O(D_{t,\overline{0}})\rightarrow\Sym\{1,...,d\},i=1,2,
\]
be the monodromies of $\widehat{f}_{1}$ and $\widehat{f}_{2}$ respectively.
Assume $\varpi_{1}(e)$ and $\varpi_{2}(e)$ are $d$-cycles, i. e. have only an
orbit and the real involutions on $P_{1}$ and $P_{2}$ are separating. The
corresponding garden of each function is an elementary garden with length $d$
on its edge and there is only one entrance and one exit in each park. Hence
the parks are equivalent, $(P_{1},f_{1})$ and $(P_{2},f_{2})$ are
topologically equivalent


\begin{thebibliography}{99}                                                                                               %


\bibitem {AG}N.~L.~Alling, N.~Greenleaf, Foundation of the theory of Klein
surfaces, Lecture Notes in Mathematics, vol. 219 (1971). 124 pp.

\bibitem {B}S. Barannikov, On the space of real polynomials without multiple
critical values, Functional Analysis and Its Applications 26 (1992) 84--90

\bibitem {C}A.~Clebsch, Zur Theory der algebraischen Funktionen, Math. Ann.,
vol. 29 (1887) 171-186.

\bibitem {CIR}A.~F.~Costa, M.~Izquierdo, G.~Riera, One-dimensional Hurwitz
spaces, modular curves, and real forms of Belyi meromorphic functions. Int. J.
Math. Math. Sci. 2008, Art. ID 609425, 18 pp.

\bibitem {CR}A.~F.~Costa, G.~Riera, One parameter families of Riemann surfaces
of genus two. Glasg. Math. J. 43 (2001), 255--268.

\bibitem {EEHS}D.~Eisenbud, N.~Elkies, J.~Harris, R.~Speiser, On the Hurwitz
scheme and its monodromy, Compositio Mathematica, vol. 77 (1991), 95--117.

\bibitem {Hu}A.~Hurwitz, \"{U}ber die Anzahl Riemannschen Fl\"{a}chen mit
gegebenen Verzweigungspunkten, Math. Ann., vol. 55 (1902), 53--66.

\bibitem {K}B.v. ~Ker\'{e}kjart\'{o}, Vorlesungen \"{u}ber Topologie. I.
Fl\"{a}chentopologie, Die Grundlehren der Mathematischen Wissenschaften,
Volume 8, Berlin: Springer-Verlag, 1923.

\bibitem {NSW}S.M.~Natanzon, B.~Shapiro and A.~Vainshtein, Topological
classification of generic real rational functions, J. Knot Theory
Ramifications, vol. 11, issue 7 (2002), 1063--1075.

\bibitem {NS}S.M.~Natanzon, S.~V.~Shadrin, Topological classification of real
meromorphic functions in general position on separating curves. Dokl. Akad.
Nauk 388 (2003), 449\^{a}\euro\textquotedblleft-451.

\bibitem {N4}S.M.Natanzon, Spaces of real meromorphic functions on real
algebraic curves, Soviet. Math.Dokl., Vol.30 (1984), 724--726.

\bibitem {N1}S.M.~Natanzon, Real meromorphic functions on real algebraic
curves, Soviet Math.Dokl., vol. 36 (1987), 425--426.

\bibitem {N2}S.M.~Natanzon, Topology of 2-dimensional coverings and
meromorphic functions on real and complex algebraic curves, Selecta Math.
Sovietica, vol. 12 (1993), 251--291.

\bibitem {N3}S.M.~Natanzon, Moduli of Riemann surfaces, real algebraic curves,
and their superanalogs, Translations of Mathematical Monographs, 225. American
Mathematical Society, Providence, RI, 2004.

\bibitem {Vak}R.~Vakil, Genus 0 and 1 Hurwitz Numbers: Recursions, Formulas,
and Graph-Theoretic Interpretations, Transactions of the American Mathematical
Society, vol. 353, issue 10 (2001), 4025--4038.
\end{thebibliography}
\end{document}